\documentclass[a4paper,10pt]{article}

\usepackage[a4paper,
            textwidth=13.0cm,
            textheight=20.0cm,
            includehead,
            includefoot]{geometry}

\usepackage{graphicx}
\usepackage{amsmath,amsfonts,amssymb}
\usepackage{mathrsfs}
\usepackage{stmaryrd}
\usepackage{booktabs}
\usepackage{enumitem}
\usepackage{xcolor}
\usepackage[colorlinks=true,linkcolor=blue,citecolor=blue,urlcolor=blue]{hyperref}
\usepackage{amsthm}

\DeclareMathOperator*{\argmin}{arg\,min}
\DeclareMathOperator{\dist}{dist}
\newcommand{\RR}{\mathbb{R}}

\makeatletter

\newcommand{\journalname}[1]{\def\@journalname{#1}}
\newcommand{\@journalname}{}

\newcommand{\titlerunning}[1]{\def\@titlerunning{#1}}
\newcommand{\authorrunning}[1]{\def\@authorrunning{#1}}
\newcommand{\@titlerunning}{}
\newcommand{\@authorrunning}{}

\newcommand{\keywordname}{\textbf{Keywords}}
\newcommand{\keywords}[1]{%
  \par\addvspace\medskipamount
  {\noindent\keywordname\enspace #1\par}
}

\newcommand{\subclassname}{\textbf{Mathematics Subject Classification (2020)}\enspace}
\newcommand{\subclass}[1]{%
  \par\addvspace\medskipamount
  {\noindent\subclassname #1\par}
}

\newcommand{\institute}[1]{\def\@institute{#1}}
\newcommand{\@institute}{}

\renewcommand{\maketitle}{%
  \begin{center}
    {\small\textbf{\@journalname}\par}
    \vspace{1.5em}
    {\LARGE\bfseries \@title \par}
    \vspace{1em}
    {\large \@author \par}
    \vspace{0.5em}
    {\small \@institute \par}
    \vspace{1.5em}
  \end{center}%
  \thispagestyle{plain}
}

\usepackage{fancyhdr}
\fancypagestyle{plain}{%
  \fancyhf{}%
  \fancyfoot[C]{\thepage}%
}

\makeatother

\numberwithin{equation}{section}

\theoremstyle{plain}
\newtheorem{theorem}{Theorem}[section]
\newtheorem{lemma}[theorem]{Lemma}
\newtheorem{proposition}[theorem]{Proposition}

\theoremstyle{definition}
\newtheorem{definition}[theorem]{Definition}
\newtheorem{assumption}[theorem]{Assumption}

\theoremstyle{remark}
\newtheorem{remark}[theorem]{Remark}

\title{Deep Centralization for the Circumcentered Reflection Method}

\titlerunning{Deep Centralization for the CRM}  
\author{Pablo Barros}
\authorrunning{P.~Barros} 

\institute{%
School of Applied Mathematics, Fundação Getulio Vargas (FGV) \\ Rio de Janeiro, Brazil \\
\texttt{pabloacbarros@gmail.com}
}

\begin{document}

\maketitle

\begin{abstract}
We introduce the extended centralized circumcentered reflection method (ecCRM), a framework for two-set convex feasibility that encompasses the classical cCRM \cite{BehlingBelloCruzIusemSantos2024} as a special case. Our method replaces the fixed centralization step of cCRM with an admissible operator $T$ and a parameter $\alpha$, allowing control over computational cost and step quality. We show that ecCRM retains global convergence, linear rates under mild regularity, and superlinearity for smooth manifolds. Numerical experiments on large-scale matrix completion indicate that deeper operators can dramatically reduce overall runtime, and tests on high-dimensional ellipsoids show that vanishing step sizes can yield significant acceleration, validating the practical utility of both algorithmic components of ecCRM.
\keywords{Convex feasibility $\cdot$ Projection methods $\cdot$ Circumcenter methods}
\subclass{65K05 $\cdot$ 65B99 $\cdot$ 90C25}
\end{abstract}

\medskip

\section{Introduction}

We study the two-set convex feasibility problem (CFP)
\begin{equation}\label{eq:CFP}
    \text{find } z \in X \cap Y,
\end{equation}
where \(X,Y\subset \mathbb{R}^n\) are closed convex sets with nonempty intersection.
Projection algorithms for such problems form a classical modeling paradigm in applied mathematics, going back at least to von Neumann's alternating projections and Cimmino’s simultaneous projections \cite{vonneumann1950functional,Cimmino1938}. 
Recently, geometrically inspired accelerations such as the circumcentered-reflection method (CRM) and its centralized variant cCRM have been shown to significantly speed up convergence, achieving superlinear rates under regularity conditions \cite{crm1,BauschkeOuyangWang2022,BehlingBelloCruzIusemSantos2024}. 

In this work, we present ecCRM, a modular advancement of cCRM. We define the iteration (Section \ref{sec:eccrm}) via an admissible operator $T$ and a relaxation parameter $\alpha$, both of which may vary at each step. In our theoretical analysis (Sections \ref{sec:conv-ecCRM} and \ref{sec:conv-order-ecCRM}), this perspective yields proofs that are novel, shorter, and simpler than the arguments in \cite{BehlingBelloCruzIusemSantos2024}, while establishing stronger results like Q-linear convergence and improving practical performance (Section \ref{sec:experiments}). Given that cCRM is already established as a highly competitive method, often surpassing classical algorithms by several orders of magnitude, our structural improvements yield further non-trivial acceleration on this tight baseline.

\subsection{Preliminaries}

We work in the standard Euclidean space \(\mathbb{R}^n\), equipped with the inner product \(\langle \cdot, \cdot \rangle\) and associated norm \(\|\cdot\|\). 
For a nonempty closed convex set \(C \subset \mathbb{R}^n\), 
\[P_C(x) := \argmin_{y\in C}\|x-y\| \quad \text{and} \quad \dist(x,C) := \inf_{y\in C}\|x-y\| \] 
denote the metric projection and distance functions, respectively. Recall the characteristic inequality $\langle z - P_C(z),\; x - P_C(z) \rangle \le 0 \ \forall \ x\in C$, equivalent to
\begin{equation}\label{eq:pythag} \|z - x\|^2 \;\ge\; \|z - P_C(z)\|^2 + \|P_C(z) - x\|^2, \qquad \forall \ z\in\RR^n, \, x\in C. \end{equation} 
Throughout, we fix sets $X, Y$ that satisfy the following assumption.
\begin{assumption}\label{ass:main}
    $X, Y \subset \mathbb{R}^n$ are closed, convex sets and \(S := X \cap Y \neq \varnothing\).
\end{assumption} 
For simplicity, we introduce the following notation:
\begin{equation}\label{def:residuals}
\delta(z) := \max \{\mathrm{dist}(z,X),\,\mathrm{dist}(z,Y)\},
\qquad
\mathfrak d(z) := \mathrm{dist}(z,S).
\end{equation}

\begin{definition}[Centralized points]
A point $z\in \RR^n$ is said to be \emph{centralized relative to $(X,Y)$} if
\[
\langle z - P_X z,\; z - P_Y z \rangle \le 0.
\]
It is said to be \emph{strictly centralized relative to $(X,Y)$} if the inequality is strict.
\end{definition}
\noindent Since the sets $X$ and $Y$ are fixed throughout, we henceforth omit the qualifier ``relative to $(X,Y)$''.

The following construction forms the core of the method and introduces a novel centralization procedure. It reveals that the specific midpoint property restricted to alternating projections in \cite[Lemma 2.2]{BehlingBelloCruzIusemSantos2024} is a special case of a broader phenomenon: centralization is robust along the entire segment connecting any $y \in Y$ to $P_X y$.

\begin{lemma}[Centralization procedure]
\label{lem:centralized-interp}
Take any $y\in Y$, any $\alpha\in (0,1)$, and define \(z := \alpha y + (1-\alpha) P_X y.\) Then:
\begin{enumerate}[label=(\roman*)]
    \item $z$ is centralized relative to $(X,Y)$.
    \item If $z$ is not strictly centralized, then $z\in Y$.
    \item If $y \in P_Y(X) \setminus S$, then $z$ is strictly centralized.
\end{enumerate}
\end{lemma}

\begin{proof}
Since $P_Xz = P_Xy$ (\cite[Prop. 3.21]{Bauschke2017}) and $\langle y - P_Y z,\; z - P_Y z \rangle \le 0$,
\begin{align*}
\langle z - P_X z,\; z - P_Y z \rangle
&= \big\langle \alpha y + (1-\alpha)P_X y - P_X y,\; z - P_Y z \big\rangle \\
&= \frac{\alpha}{1-\alpha} \langle y - z,\; z - P_Y z \rangle \\
&= \frac{\alpha}{1-\alpha} \big( \langle y - P_Y z,\; z - P_Y z \rangle - \|z - P_Y z\|^2 \big) \le 0,
\end{align*}
with equality only if $z = P_Y z \in Y$, which yields (i) and (ii). If $y \in P_Y(X) \setminus S$, then $y = P_Y x$ with $x \in X$ and $y \neq P_X y \in X$. Thus
\[
\langle P_X y - y,\, x - y \rangle
=
\langle P_X y - y,\, x - P_X y \rangle + \|P_X y - y\|^2
> 0,
\]
because $x-y = (x-P_X y) + (P_X y-y)$ and $\langle y - P_X y,\, x - P_X y \rangle \le 0.$ Then
\[
\langle z - y,\, x - y \rangle
=
\big\langle \alpha y + (1-\alpha)P_X y - y,\, x - y \big\rangle
=
(1-\alpha)\,\langle P_X y - y,\, x - y \rangle
> 0.
\]
Hence $z \notin Y$ and by item (ii), $z$ must be strictly centralized, proving (iii).
\end{proof}

Recall that a sequence $(z_k)_{k \ge 0}$ is \emph{Fejér monotone with respect to $S$} if
\begin{equation}\label{eq:Fejer}
\|z_{k+1}-s\| \le \|z_k-s\|
\qquad
\forall\, s\in S,\ \forall\,k \ge 0.
\end{equation}
We collect below some well-known implications of this property.

\begin{proposition}\label{prop:fejer}
If $(z_k)\subset\mathbb R^n$ is Fejér monotone with respect to $S$, then:
\begin{enumerate}[label=(\roman*)]
\item If $\delta(z_k) \to 0$, then $(z_k)$ converges to a point in $S$.
\item If $(z_k)$ converges to a point $\bar{z}$, then, for every $k \ge 0$, \(\mathrm{dist}(z_k,S) \ge \frac{1}{2}\,\|z_k - \bar{z}\|.\)
\item If the scalar sequence $\bigl(\mathrm{dist}(z_k,S)\bigr)_{k \ge 0}$ converges Q-linearly (resp. superlinearly) to $0$, then $(z_k)_{k \ge 0}$ converges R-linearly (resp. superlinearly) to a point $\bar z \in S$, with the same asymptotic constant.
\end{enumerate}
\end{proposition}

\begin{proof}
For (i), fix $s\in S$. By \eqref{eq:Fejer}, $\|z_k-s\| \le \|z_0-s\|$, so $(z_k)$ is bounded. Hence a subsequence $(z_{k_j})$ converges to some
$\bar z$. Since $\delta$ is continuous, $\delta(z_{k_j})\to 0$ implies \(\delta(\bar z) = 0,\) so $\bar z\in X\cap Y = S$. Applying \eqref{eq:Fejer} with $s=\bar z$ shows that $(\|z_k-\bar z\|)_k$ is
nonincreasing. On the other hand, \(\lim_{j\to\infty} \|z_{k_j}-\bar z\| = 0,\) so necessarily $\|z_k-\bar z\|\to 0$. Thus $z_k\to \bar z\in S$. For (ii), see, e.g., \cite[Proposition~2]{Barros3pm}. For (iii), see, e.g., \cite[Theorem~5.12]{Bauschke2017} (resp. \cite[Proposition~3.12]{BehlingBelloCruzIusemSantos2024}.)
\end{proof}

\section{The ecCRM Algorithm}\label{sec:eccrm}

\begin{definition}[Circumcenter]\label{def:circumcenter}
Let $z,v,w \in \mathbb{R}^n$ be given.
A point $c \in \mathbb{R}^n$ is called the \emph{circumcenter} of $\{z,v,w\}$ if
\begin{enumerate}[label=(\roman*)]
    \item $\|c - z\| = \|c - v\| = \|c - w\|$, and
    \item $c \in \operatorname{aff}\{z,v,w\}
    := \{u \in \mathbb{R}^n \mid u = z + \alpha (v-z) + \beta (w-z),\ \alpha,\beta \in \mathbb{R}\}$.
\end{enumerate}
When it exists, it is unique and denoted by \(\operatorname{circ}(z,v,w).\)
\end{definition}

\begin{definition}[PCRM operator]\label{def:PCRM}
Let $X,Y \subset \mathbb{R}^n$ be closed convex sets with reflection operators \(R_X := 2P_X - \operatorname{Id}, \ R_Y := 2P_Y - \operatorname{Id}\).
The \emph{parallel circumcentered-reflection} operator
$\mathrm{PCRM} : \mathbb{R}^n \to \mathbb{R}^n$ is defined by
\[
    \mathrm{PCRM}(z) 
    := \operatorname{circ}\bigl(z, R_X(z), R_Y(z)\bigr),
    \qquad z \in \mathbb{R}^n.
\]
\end{definition}
\noindent
Notice that if $z \in Y$ (resp. $z \in X$) then $\mathrm{PCRM}(z)$ is just $P_Xz$ (resp. $P_Yz$). We write only $\mathrm{PCRM}$ for $\mathrm{PCRM}_{X,Y}$ since $X,Y$ are fixed.

\begin{definition}[ecCRM operator]
\label{def:ecCRMoperator}
An operator $T:\RR^n\to\RR^n$ is called \emph{admissible} if it maps to $Y$ and is quasi-nonexpansive relative to $S$, i.e.,
\begin{equation}\label{def:admissible}
    \mathrm{Im}\,T \subset Y
\quad\text{and}\quad
\|Tz - s\| \le \|z - s\| \quad \forall\, z\in \RR^n,\ \forall\, s\in S.
\end{equation}
Now fix $\alpha\in(0,1)$.  The $\alpha$-\emph{centralizer} operator associated with $T$ is
\[
N^\alpha := \alpha\, T + (1-\alpha)\, P_X T.
\]
Finally, the corresponding \emph{ecCRM operator} is
\[
\mathscr{C}^\alpha := \mathrm{PCRM}_{X,Y} \circ N^\alpha.
\]
\end{definition}

\begin{definition}[ecCRM method]
\label{def:ecCRM}
Fix an admissible operator $T$ and a sequence \((\alpha_k) \subset (0,1)\). Starting at $z \in \RR^n,$ ecCRM iterates by
\[
z_{k+1} = \mathscr{C}^{\alpha_k} z_k, \qquad z_0 = z.
\]
\end{definition}

\noindent
Although Definition \ref{def:ecCRM} fixes $T$ for notational simplicity, the convergence analysis holds if we choose a different admissible operator $T_k$ at every iteration. 

\begin{remark}[Computational note and examples]
Since \(N^\alpha z\) lies on the segment connecting \(Tz\) and \(P_X(Tz)\), we have \(P_X(N^\alpha z) = P_X(Tz)\). Thus, once \(Tz\) is evaluated, the method requires only \emph{two} additional projections to form the reflections needed by PCRM: \(P_X(Tz)\) and \(P_Y(N^\alpha z)\).
In general, any operator of the form \(T = P_Y \circ J\) is admissible, provided \(J\) is quasi-nonexpansive with respect to \(S\).
Specific instances include the basic \(T = P_Y\) (where \(J=\operatorname{Id}\), total 3 projections); the standard cCRM kernel \(T = P_Y P_X\) (where \(J=P_X\), 4 projections); and the ``deep'' kernel \(T = P_Y P_X P_Y\) (where \(J=P_X P_Y\), 5 projections).
\end{remark}

\section{Convergence of ecCRM}\label{sec:conv-ecCRM}

We now recall a key structural result for centralized circumcenters from \cite{BehlingBelloCruzIusemSantos2024}.

\begin{lemma}[{\cite[Lemma~2.3]{BehlingBelloCruzIusemSantos2024}}]
\label{lem:2.3ccrm}
Let $z\in\RR^n$ be centralized. Consider
\begin{align*}
    S_X^z &:= \big\{w \in \RR^n \,\big|\, \langle w - P_X(z),\, z - P_X(z) \rangle \le 0\big\},
    \\S_Y^z &:= \big\{w \in \RR^n \,\big|\, \langle w - P_Y(z),\, z - P_Y(z) \rangle \le 0\big\}.
\end{align*}
Then
\[
\mathrm{PCRM}_{X,Y}(z) = P_{S_X^z \cap S_Y^z}(z).
\]
\end{lemma}

\begin{lemma}\label{lem:decreasedelta}
If $z$ is centralized, then for every $s\in S$,
\begin{equation}\label{eq:decreasedelta}
\|s-\mathrm{PCRM}(z)\|^2
\;\le\;
\|s-z\|^2 - \max\{\delta(z),\delta(\mathrm{PCRM}(z))\}^2.
\end{equation}
\end{lemma}

\begin{proof}
Let $z$ be centralized. By Lemma~\ref{lem:2.3ccrm}, take the sets \(S_X^z,\; S_Y^z\) with
\[
\mathrm{PCRM}(z) = P_H(z),
\qquad
H := S_X^z \cap S_Y^z \supset S.
\]
Fix $s\in S\subset H$. The projection inequality \eqref{eq:pythag} applied to $H$ gives
\begin{equation}\label{eq:proj-pyth-dec1}
\|s-\mathrm{PCRM}(z)\|^2
\le
\|s-z\|^2 - \|z-\mathrm{PCRM}(z)\|^2.
\end{equation}
By construction, \(P_Y(z) = P_{S_Y^z}(z).\)
Since $\mathrm{PCRM}(z)\in S_Y^z$, another application of \eqref{eq:pythag} (now to $S_Y^z$) yields
\begin{align*}
    \|z-\mathrm{PCRM}(z)\|^2 
\;&\ge\;
\|z-P_Y(z)\|^2 + \|\mathrm{PCRM}(z)-P_Y(z)\|^2
\\\;&\ge\;
\max\big\{\|z-P_Y(z)\|^2,\;\|\mathrm{PCRM}(z)-P_Y(z)\|^2\big\}.
\end{align*}
Since $P_Y(z)\in Y$, we have \(\|\mathrm{PCRM}(z)-P_Y(z)\| \ge\mathrm{dist}(\mathrm{PCRM}(z),Y),\)
so
\[
\|z-\mathrm{PCRM}(z)\|
\;\ge\;
\max\big\{\mathrm{dist}(z,Y),\,\mathrm{dist}(\mathrm{PCRM}(z),Y)\big\}.
\]
Repeating the same argument with $S_X^z$ and $P_X(z)$ and combining the bounds:
\begin{equation}\label{step:dec1}
\|z-\mathrm{PCRM}(z)\|
\;\ge\;
\max\{\delta(z),\delta(\mathrm{PCRM}(z))\}.
\end{equation}
Finally, substituting \eqref{step:dec1} into \eqref{eq:proj-pyth-dec1} yields precisely \eqref{eq:decreasedelta}.
\end{proof}

We may now state and prove the global convergence result for ecCRM.

\begin{theorem}[Global convergence of ecCRM]\label{thm:ecCRM-conv}
Let $T$ be an admissible operator (see \eqref{def:admissible}) and let $(\alpha_k)\subset(0,1)$ be arbitrary. Consider the ecCRM iterates \((z_k)_{k \ge 0}\) starting from some $z_0\in\RR^n$. Then $(z_k)$ is Fejér monotone with respect to $S$ and converges to a point in $S$.
\end{theorem}

\begin{proof}
Fix $s\in S$. By Lemma~\ref{lem:centralized-interp}, $N^{\alpha_k} z_k$ is centralized with respect to $(X,Y)$, so Lemma~\ref{lem:decreasedelta} applies with $z = N^{\alpha_k} z_k$ and $\mathrm{PCRM}(z) = z_{k+1} = \mathscr{C}^{\alpha_k} z_k$:
\begin{equation}\label{eq:Fejer-from-N}
\|s-z_{k+1}\|^2
\le
\|s-N^{\alpha_k} z_k\|^2
-
\max\big\{\delta(N^{\alpha_k} z_k),\,\delta(z_{k+1})\big\}^2.
\end{equation}
Using the definition of $N^{\alpha_k} z_k$ and convexity of the norm,
\[
\|s-N^{\alpha_k} z_k\|
\le
\alpha_k \|s-Tz_k\| + (1-\alpha_k)\|s-P_XTz_k\|.
\]
Since $s\in X$, we have $\|s-P_XTz_k\|\le \|s-Tz_k\| \le \|s-z_k\|$, hence
\begin{equation}\label{eq:sNa-sz-iter}
\|s-N^{\alpha_k} z_k\| \le \|s-z_k\|.
\end{equation}
Combining \eqref{eq:Fejer-from-N} and \eqref{eq:sNa-sz-iter} implies
\begin{equation}\label{eq:iter-ineq}
\|s-z_{k+1}\|^2
\le
\|s-z_k\|^2 - \delta(z_{k+1})^2
\qquad\forall \ k \ge 0.
\end{equation}
In particular, $(z_k)$ is Fejér monotone with respect to $S$. Summing \eqref{eq:iter-ineq} over $k \ge 0$ yields
\[
\sum_{k=1}^\infty \delta(z_k)^2 \le \|s-z_0\|^2,
\]
so $\delta(z_k)\to 0$ as $k\to\infty$. Now, Proposition~\ref{prop:fejer}~(i) shows that $(z_k)$ converges to some point $\bar z\in S$. This completes the proof.
\end{proof}

\section{Convergence order of ecCRM}\label{sec:conv-order-ecCRM}

To analyze the local convergence rate of the sequence generated by ecCRM, we require a classical local \emph{error bound} condition, also known as
linear regularity.

\begin{definition}[Error bound]\label{def:EB}
We say that sets $X, \, Y \subset \mathbb{R}^n$ satisfy a local
error bound condition if for some point $\bar z \in X \cap Y$ there exist
a real number $\omega \in (0,1)$ and a neighborhood $V$ of $\bar z$ such that for all $z \in V$:
\begin{equation}\tag{EB}\label{eq:EB}
    \omega\,\mathrm{dist}(z,X \cap Y)
    \;\le\;
    \max\{\mathrm{dist}(z,X),\mathrm{dist}(z,Y)\}.
\end{equation}
\end{definition}
\noindent

\begin{lemma}[Linear advance by PCRM]\label{lem:pcrm-rate}
    Assume the error bound condition~\eqref{eq:EB} and let $\beta := \sqrt{1-\omega^2}$. Then, for any centralized $z \in V$,
    \[
    \mathfrak d(\mathrm{PCRM}(z)) \;\le\; \beta\, \mathfrak d(z).
    \]
\end{lemma}
\begin{proof}
Let $s := P_S(z)$. Note that $\mathfrak d(\mathrm{PCRM}(z)) \le \|\mathrm{PCRM}(z)-s\|$ and $\mathfrak d(z) = \|z-s\|$. Applying Lemma~\ref{lem:decreasedelta} with this $s$ then gives
\begin{equation}\label{eq:contraction-base}
\mathfrak d(\mathrm{PCRM}(z))^2
\;\le\;
\mathfrak d(z)^2 - \delta(z)^2.
\end{equation}
By the error bound condition~\eqref{eq:EB}, \(\delta(z) \ge \omega\,\mathfrak d(z).\)
Plugging into \eqref{eq:contraction-base} yields
\[
\mathfrak d(\mathrm{PCRM}(z))^2
\;\le\;
\mathfrak d(z)^2 - \omega^2 \mathfrak d(z)^2
=
(1-\omega^2)\,\mathfrak d(z)^2
=
\beta^2\,\mathfrak d(z)^2,
\]
and taking square roots gives \(\mathfrak d(\mathrm{PCRM}(z)) \le \beta\,\mathfrak d(z).\)
\end{proof}

\begin{lemma}[Single-step rate]
\label{lem:rates}
Assume the error bound condition~\eqref{eq:EB} holds and set
\(\beta := \sqrt{1-\omega^2}.\) Then, for every \(z\in V\),
\begin{enumerate}[label=(\roman*)]
    \item  we have \[
\mathfrak d(\mathscr{C}^\alpha z)
\;\le\; \beta (\alpha+(1-\alpha)\beta ) \, \mathfrak d(Tz);
\]
\item moreover, if \(N^\alpha z\) is not strictly centralized, then 
\[
\mathfrak d(\mathscr{C}^\alpha z)
\;\le\;
\frac{\alpha \beta}{1-(1-\alpha)\beta} \,\mathfrak d(Tz).
\]
\end{enumerate}
\end{lemma}

\begin{proof}
Recall that \(N^\alpha z\) is centralized and
\(\mathscr{C}^\alpha z = \mathrm{PCRM}(N^\alpha z).\)
By Lemma~\ref{lem:pcrm-rate},
\begin{equation}\label{eq:C-alpha-vs-N}
\mathfrak d(\mathscr{C}^\alpha z)
\;\le\;
\beta\,\mathfrak d(N^\alpha z).
\end{equation}
Next, by convexity of \(\mathfrak d\), we get
\(\mathfrak d(N^\alpha z)\le
\alpha\,\mathfrak d(Tz) + (1-\alpha)\,\mathfrak d(P_XTz).
\)
Combining with \eqref{eq:C-alpha-vs-N} yields
\begin{equation}\label{eq:C-alpha-ineq-raw}
\mathfrak d(\mathscr{C}^\alpha z)
\;\le\;
\beta\big(\alpha\,\mathfrak d(Tz) + (1-\alpha)\,\mathfrak d(P_XTz)\big).
\end{equation}
Notice \(Tz \in Y\) is centralized and \(\mathrm{PCRM}(Tz) = P_XTz,\) hence by Lemma~\ref{lem:pcrm-rate},
\(\mathfrak d(P_XTz) \le \beta \, \mathfrak d(Tz).
\) Inequality \eqref{eq:C-alpha-ineq-raw} then gives (i):
\[
\mathfrak d(\mathscr{C}^\alpha z)
\;\le\;
\beta\big(\alpha\,\mathfrak d(Tz) + (1-\alpha)\beta \, \mathfrak d(Tz)\big) = \beta (\alpha+(1-\alpha)\beta ) \, \mathfrak d(Tz).
\]
Now if \(N^\alpha z\) is not strictly centralized, Lemma~\ref{lem:centralized-interp}~(ii) implies \(N^\alpha z \in Y,\) then \(\mathscr{C}^\alpha z = P_XN^\alpha z = P_XTz\), so \eqref{eq:C-alpha-ineq-raw} is
\[
\mathfrak d(\mathscr{C}^\alpha z)
\;\le\;
\beta\big(\alpha\,\mathfrak d(Tz) + (1-\alpha)\,\mathfrak d(\mathscr{C}^\alpha z)\big).
\]
Rearranging terms, we get the desired estimate (ii).
\end{proof}

We are now ready to show the linear convergence result for our method.

\begin{theorem}[Linear convergence of ecCRM]\label{thm:ecCRM-Rlinconv}
Let $T$ be an admissible operator (see \eqref{def:admissible}) and 
$(\alpha_k)\subset(0,1)$ be arbitrary. Assume the error bound~\eqref{eq:EB} 
holds with constant $\omega\in(0,1)$, and set
\[
\beta := \sqrt{1-\omega^2},\qquad
\bar\alpha := \limsup_{k \to \infty} \alpha_k,\qquad
\rho := \beta\bigl(\bar \alpha + (1-\bar \alpha)\beta\bigr).
\]
For the ecCRM iterates $(z_k)$ from $z_0\in\RR^n$, the distance sequence
$(\mathfrak d(z_k))$ converges Q-linearly to $0$ with asymptotic constant at
most $\rho$, and $(z_k)$ converges R-linearly to a point in $S$ with the same
bound. Moreover, $(z_k)$ converges Q-linearly with constant at most
$(1+\omega^2/4)^{-1/2}$.
\end{theorem}

\begin{proof}
By Theorem~\ref{thm:ecCRM-conv}, $(z_k)$ converges to some $\bar z\in S$, so there exists
$k_0$ such that $z_k\in V$ for all $k\ge k_0$, where $V$ is the neighborhood in
\eqref{eq:EB}. For $k\ge k_0$,
\[
\mathfrak d(z_{k+1})
\le
\beta\bigl(\alpha_k + (1-\alpha_k)\beta\bigr)\,\mathfrak d(Tz_k)
\]
by Lemma~\ref{lem:rates}\,(i). Since $T$ is admissible, $\mathfrak d(Tz_k)\le \mathfrak d(z_k)$, hence
\[
\mathfrak d(z_{k+1})
\le
\beta\bigl(\alpha_k + (1-\alpha_k)\beta\bigr)\,\mathfrak d(z_k),
\qquad k\ge k_0.
\]
Taking $\limsup$ in $k$ and using the definition of $\rho$ yields the first claim. By Proposition~\ref{prop:fejer}~(iii), $(z_k)$ converges R-linearly to $\bar z\in S$
with asymptotic constant bounded by $\rho$. For Q-linear convergence of $(z_k)$, apply \eqref{eq:iter-ineq} with $s=\bar z$:
\[
\|z_{k+1}-\bar z\|^2
\le
\|z_k-\bar z\|^2 - \delta(z_{k+1})^2.
\]
Using the error bound and Proposition~\ref{prop:fejer}~(ii) for $z_{k+1} \in V$,
\[
\delta(z_{k+1}) \ge \omega\,\mathfrak d(z_{k+1})
\ge \frac{\omega}{2}\|z_{k+1}-\bar z\|,
\]
so
\[
\|z_{k+1}-\bar z\|^2
\le
\|z_k-\bar z\|^2 - \frac{\omega^2}{4}\|z_{k+1}-\bar z\|^2 \, \implies \, \Bigl(1+\frac{\omega^2}{4}\Bigr)\|z_{k+1}-\bar z\|^2 \le \|z_k-\bar z\|^2,
\]
ensuring $(z_k)$ converges Q-linearly to $\bar z$ with constant at most $(1+\omega^2/4)^{-1/2}$.
\end{proof}
\noindent
We highlight that Q-linear convergence for cCRM iterates is not present in \cite{BehlingBelloCruzIusemSantos2024}.

\begin{remark}[Improved rates via modularity]\label{rem:deep-kernel}
We benchmark against the standard cCRM rate from \cite[Theorem~3.10]{BehlingBelloCruzIusemSantos2024}, bounded by $\beta_{\mathrm{cCRM}} := \tfrac{1}{2}\beta^2(1+\beta)$. The modularity of ecCRM allows us to strictly improve this constant in two ways. First, with the standard kernel $T=P_YP_X$, notice that $Tz = P_Y(P_X z)$, implying $\mathfrak d(Tz) \le \beta \mathfrak d(z)$ by Lemma~\ref{lem:pcrm-rate}. Combining this with vanishing steps $\alpha_k \to 0$ in Lemma~\ref{lem:rates} yields an asymptotic constant of $\beta^3 < \beta_{\mathrm{cCRM}}$. Second, the deep kernel $T_{\mathrm{deep}} = P_Y P_X P_Y$ involves two alternations, yielding $\mathfrak d(T_{\mathrm{deep}}z) \le \beta^2 \mathfrak d(z)$. Substituting this into Lemma~\ref{lem:rates} with fixed $\alpha_k \equiv 1/2$ yields a rate of $\beta \cdot \beta_{\mathrm{cCRM}} < \beta_{\mathrm{cCRM}}$. Thus, both dynamic centralization and deeper operators enforce strictly stronger contraction than the classical method.
\end{remark}

We now investigate conditions under which ecCRM exhibits superlinear convergence
near a smooth intersection point of \(X\) and \(Y\). 
\begin{assumption}\label{ass:super-smooth}
In addition to Assumption \ref{ass:main},
    \begin{enumerate}[label=(\roman*)]
        \item $\mathrm{int}(S)$ is nonempty;
        \item in a neighborhood of a point \(\bar z\in \partial X\cap\partial Y\) (taken as the limit point of $(z_k)$),
the boundaries \(\partial X\) and \(\partial Y\) are \(C^1\) manifolds of
dimension \(n-1\).
    \end{enumerate}
\end{assumption}
This is exactly the geometric setting used in
\cite[Section~3.3]{BehlingBelloCruzIusemSantos2024} for the superlinear analysis of cCRM. Condition
\(\mathrm{int}(S)\neq\varnothing\) also implies the local error bound~\eqref{eq:EB} around \(\bar z\).

We will use the following tangency estimate from \cite[Lemma~2.10]{Barros3pm}.

\begin{lemma}[Tangency limit to manifold]\label{lemmanif}
Let $M \subset \mathbb{R}^n$ be a differentiable manifold of dimension $n-1$
and let $\bar z \in M$. If $(q_k) \subset M$ and $(z_k) \subset \mathbb{R}^n$
satisfy
\[
q_k \to \bar z, \qquad z_k \to \bar z, \qquad z_k \in T_M(q_k), \qquad z_k \neq q_k,
\]
where $T_M(q_k)$ denotes the affine tangent hyperplane to $M$ at $q_k$, then
\[
\lim_{k \to \infty} 
\frac{\operatorname{dist}(z_k,M)}{\|z_k - q_k\|} = 0.
\]
\end{lemma}
This implies the superlinear version of Lemma \ref{lem:pcrm-rate}:

\begin{lemma}[Superlinear advance by PCRM]
\label{lem:PCRM-superlinear-ecCRM}
Let Assumption \ref{ass:super-smooth} hold, and let
$(w_k)$ be a sequence of strictly centralized points converging to $\bar z \in S$.
Then
\[
\lim_{k\to\infty} \frac{\mathfrak d(\mathrm{PCRM}(w_k))}{\mathfrak d(w_k)} = 0.
\]
\end{lemma}

\begin{proof}
    Set $c_k := \mathrm{PCRM}(w_k)$ and $\mathfrak d_k := \mathfrak d(w_k) = \mathrm{dist}(w_k,S)$.
By strict centralization and Lemma~\ref{lem:2.3ccrm}, we can write \(c_k = P_{H_X^k \cap H_Y^k}(w_k),\) where
\begin{align*}
    H_X^k &:= \bigl\{z\in\RR^n \;\big|\; \langle z - P_X w_k,\; w_k - P_X w_k\rangle = 0\bigr\},
\\H_Y^k &:= \bigl\{z\in\RR^n \;\big|\; \langle z - P_Y w_k,\; w_k - P_Y w_k\rangle = 0\bigr\}
\end{align*}
are supporting hyperplanes through $P_X w_k$ and $P_Y w_k$, respectively. Since $w_k\to\bar z$, we have \(P_X w_k \to \bar z, \ P_Y w_k \to \bar z.\)
By local $C^1$ smoothness of $\partial X$ and $\partial Y$, for all large $k$ the hyperplanes
$H_X^k$ and $H_Y^k$ coincide with the tangent spaces
\[
H_X^k = T_{\partial X}(P_X w_k),
\qquad
H_Y^k = T_{\partial Y}(P_Y w_k).
\]
Now apply Lemma \ref{lemmanif} to the manifolds
$\partial X$ and $\partial Y$. With
\[
q_k^X := P_X w_k \in \partial X, \quad z_k := c_k \in H_X^k = T_{\partial X}(q_k^X),
\]
and $q_k^X,z_k\to\bar z$, that lemma yields
\[
\lim_{k\to\infty}
\frac{\mathrm{dist}(c_k,X)}{\|c_k - P_X w_k\|} = 0.
\]
Since $H_X^k$ and $H_Y^k$ are boundaries of supporting half-spaces containing $S$,
\[
\mathfrak d_k
=
\mathrm{dist}(w_k,S)
\;\ge\;
\mathrm{dist}(w_k, H_X^k\cap H_Y^k)
=
\|w_k - c_k\|.
\]
Since $c_k \in H_X^k$ and $P_X w_k = P_{H_X^k}w_k,$
\[
\|w_k - c_k\|^2
=
\|w_k - P_X w_k\|^2 + \|c_k - P_X w_k\|^2,
\]
which implies \(\|c_k - P_X w_k\|
\le
\|w_k - c_k\|
\le
\mathfrak d_k.\) Combining these facts,
\[
\frac{\mathrm{dist}(c_k,X)}{\mathfrak d_k}
=
\frac{\mathrm{dist}(c_k,X)}{\|c_k - P_X w_k\|}
\cdot
\frac{\|c_k - P_X w_k\|}{\mathfrak d_k}
\;\longrightarrow\;
0,
\]
The same reasoning gives \(\mathrm{dist}(c_k,Y)/\mathfrak d_k \to 0.\)
Finally, since \eqref{eq:EB} holds in a neighborhood of $\bar z$, we have
for all large $k$:
\[
\omega\,\mathrm{dist}(c_k,S)
\;\le\;
\max\{\mathrm{dist}(c_k,X),\mathrm{dist}(c_k,Y)\}.
\]
Dividing by $\mathfrak d_k$ and letting $k\to\infty$ yields
\begin{align*}
    \limsup_{k\to\infty}
\frac{\mathfrak d(\mathrm{PCRM}(w_k))}{\mathfrak d(w_k)}
&=
\limsup_{k\to\infty}
\frac{\mathrm{dist}(c_k,S)}{\mathfrak d_k}
\\\;&\le\;
\frac{1}{\omega}
\limsup_{k\to\infty}
\max\!\left\{
\frac{\mathrm{dist}(c_k,X)}{\mathfrak d_k},
\frac{\mathrm{dist}(c_k,Y)}{\mathfrak d_k}
\right\} = 0,
\end{align*}
which is the desired superlinear advance.
\end{proof}

Combining Lemma~\ref{lem:PCRM-superlinear-ecCRM} with Lemma~\ref{lem:centralized-interp}~(iii) leads to the
following superlinear result for ecCRM with “good” kernels $T$.

\begin{theorem}[Superlinear by $\mathrm{Im}\,T\subset P_Y(X)$]
\label{thm:ecCRM-superlinear-T}
Let $T$ be an admissible
operator such that $\mathrm{Im}\,T\subset P_Y(X)$, and let $(\alpha_k)\subset(0,1)$
be arbitrary. Consider the ecCRM iterates \((z_k)\) starting from some $z_0\in\RR^n$, and let $z_k\to\bar z$ by Theorem \ref{thm:ecCRM-conv}. Let Assumption \ref{ass:super-smooth} hold. Then the distance sequence
\(\bigl(\mathfrak d(z_k)\bigr)_{k\ge 0}\) converges to $0$ superlinearly, and the iterates $(z_k)$ themselves converge to $\bar z$ superlinearly.
\end{theorem}

\begin{proof}
For each $k$, set
\(w_k := N^{\alpha_k} z_k.\) By Lemma~\ref{lem:centralized-interp}~(iii) applied with $y = Tz_k$, for every $k$ we have either $w_k$ strictly centralized, or $Tz_k \in S.$ The latter yields $\mathscr{C}^{\alpha_k} z_k \in S$, hence $(z_k)$ becomes stationary and the conclusion is obvious. Now we assume that $w_k$ is strictly centralized.

By \eqref{eq:sNa-sz-iter} we have \(\mathfrak d(w_k) \le \mathfrak d(z_k).\) By Lemma~\ref{lem:PCRM-superlinear-ecCRM}, we obtain
\[
\lim_{k\to\infty}
\frac{\mathfrak d(z_{k+1})}{\mathfrak d(z_k)}
\le
\lim_{k\to\infty}
\frac{\mathfrak d(z_{k+1})}{\mathfrak d(w_k)}
=
\lim_{k\to\infty}
\frac{\mathfrak d(\mathrm{PCRM}(w_k))}{\mathfrak d(w_k)}
= 0.
\]
Finally, by Proposition~\ref{prop:fejer}\,(iii), $(z_k)$ converges superlinearly to $\bar z\in S$.
\end{proof}

When $N^\alpha z$ is not strictly centralized, the advance is guaranteed by the “linear” estimate in Lemma~\ref{lem:rates}. This is used in the next lemma.

\begin{theorem}[Superlinear by vanishing step sizes]
\label{thm:ecCRM-superlinear-alpha}
Let $T$ be any admissible
operator, and let $(\alpha_k)\subset(0,1)$
satisfy \(\alpha_k \to 0.\) Consider the ecCRM iterates \((z_k)_{k \ge 0}\) starting from some $z_0\in\RR^n$, and let $z_k\to\bar z$ by Theorem \ref{thm:ecCRM-conv}. Let Assumption \ref{ass:super-smooth} hold with this $\bar z.$ Then the distance sequence \(\bigl(\mathfrak d(z_k)\bigr)_{k\ge 0}\) converges to $0$ superlinearly, and the iterates $(z_k)$ themselves converge to $\bar z$ superlinearly.
\end{theorem}

\begin{proof}
Assume $z_k \notin S$ for all $k$, otherwise the claim is obvious. Let $w_k := N^{\alpha_k}z_k$. By \eqref{eq:sNa-sz-iter}, we have $\mathfrak d(w_k) \le \mathfrak d(z_k)$.
We examine the ratio $\mathfrak d(z_{k+1})/\mathfrak d(z_k)$ in two cases.
If $w_k$ is strictly centralized, Lemma~\ref{lem:PCRM-superlinear-ecCRM} applies (since $w_k \to \bar z$), yielding
\[
\frac{\mathfrak d(z_{k+1})}{\mathfrak d(z_k)}
\;\le\;
\frac{\mathfrak d(\mathrm{PCRM}(w_k))}{\mathfrak d(w_k)}
\;\xrightarrow{k\to\infty}\; 0.
\]
If $w_k$ is not strictly centralized, Lemma~\ref{lem:rates}\,(ii) and the admissibility of $T$ imply
\[
\frac{\mathfrak d(z_{k+1})}{\mathfrak d(z_k)}
\;\le\;
\frac{\alpha_k\beta}{1-(1-\alpha_k)\beta}
\;\xrightarrow{k\to\infty}\; 0,
\]
since $\alpha_k \to 0$.
In either case, $\limsup_{k\to\infty} \mathfrak d(z_{k+1})/\mathfrak d(z_k) = 0$. By Proposition~\ref{prop:fejer}\,(iii), the sequence $(z_k)$ converges superlinearly to $\bar z$.
\qed
\end{proof}

\section{Numerical Experiments}
\label{sec:experiments}

We evaluate ecCRM on two benchmarks: rank-deficient matrix completion and high-dimensional ellipsoids. In all tests, algorithms were terminated when the feasibility gap satisfied $\delta(z_k) \le \varepsilon$. Experiments were conducted in Python on a standard workstation; code is available at \url{https://github.com/pabl0ck/cfp-bap-lab}. 

\paragraph{Benefit of Deep Kernels: Matrix Completion.} 
We first consider the positive semidefinite (PSD) matrix completion problem:
\begin{equation}
    \text{find } Z \in \mathcal{S}_+^n \cap \{M \in \mathbb{R}^{n \times n} \mid M_{ij} = A_{ij} \text{ for } (i,j) \in \Omega\}.
\end{equation}
We generated 10 random instances ($n=100$, rank $5$, 40\% observed entries). This geometry lacks the property $\text{int} (S) \neq \varnothing$ required for superlinear rates. We compared the standard cCRM kernel ($T = P_Y P_X$, 4 projections) against the modular ``Deep'' kernel ($T = P_Y P_X P_Y$, 5 projections) with a tolerance of $10^{-2}$. Notably, the Deep kernel outperformed the standard variant across all three tested step sizes. Despite the higher per-iteration cost, it reduced total runtime by approximately \textbf{9.1\%} and iteration count by \textbf{9.4\%} at the optimal step size $\alpha=0.5$.

\begin{table}[h]
\centering
\caption{Runtime sensitivity to step size $\alpha$. $\alpha=0.5$ minimizes time for both methods.}
\label{tab:alpha}
\setlength{\tabcolsep}{5pt}
\begin{small}
\begin{tabular}{l|ccc|ccc}
\toprule
 & \multicolumn{3}{c|}{\textbf{Time (s)}} & \multicolumn{3}{c}{\textbf{Iterations}} \\
\textbf{Kernel} & $\alpha=0.25$ & $\mathbf{\alpha=0.5}$ & $\alpha=0.75$ & $\alpha=0.25$ & $\mathbf{\alpha=0.5}$ & $\alpha=0.75$ \\
\midrule
cCRM      & 29.24 & 21.01 & 31.59 & 4363 & 3128 & 4717 \\
Deep      & \textbf{26.89} & \textbf{19.25} & \textbf{29.39} & \textbf{3962} & \textbf{2834} & \textbf{4365} \\
\bottomrule
\end{tabular}
\end{small}
\end{table}

\paragraph{Acceleration via Vanishing Step Size: Intersection of Ellipsoids.} 
To validate the impact of the relaxation parameter, we consider the intersection of two nearly tangent, anisotropic ellipsoids in $\mathbb{R}^{2000}$ (condition number $\approx 20$). Since the intersection has nonempty interior, both methods are expected to converge linearly. Using $T=P_YP_X$ and a tighter tolerance of $10^{-12}$, we compare the fixed step $\alpha=0.5$ (cCRM) against ecCRM with the vanishing schedule $\alpha_k = (k+2)^{-1}$.

\begin{figure}[h]
    \centering
    \includegraphics[width=0.7\textwidth]{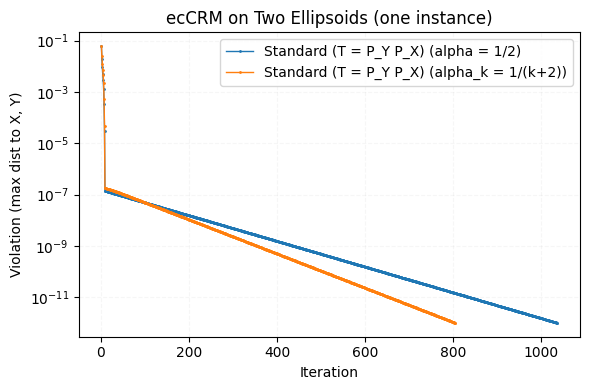}
    \caption{Convergence on $\mathbb{R}^{2000}$ Ellipsoids.}
    \label{fig:ellipsoid_conv}
\end{figure}

As illustrated in Fig.~\ref{fig:ellipsoid_conv}, the dynamic strategy achieves a noticeably steeper descent slope. This geometric advantage translates to a \textbf{15.1\%} reduction in iterations and a \textbf{20.3\%} reduction in total runtime (Table \ref{tab:ellipsoids}), confirming that dynamic centralization significantly improves the asymptotic rate constant.

\begin{table}[h]
\centering
\caption{Comparison on $\mathbb{R}^{2000}$ Ellipsoids (Avg.\ 10 runs).}
\label{tab:ellipsoids}
\setlength{\tabcolsep}{6pt}
\begin{small}
\begin{tabular}{lccc}
\toprule
\textbf{Method} & \textbf{Iters} & \textbf{Time (s)} & \textbf{Final Viol} \\
\midrule
cCRM ($\alpha=0.5$) & 1020 & 13.64 & $9.9 \times 10^{-13}$ \\
\textbf{ecCRM ($\mathbf{\alpha_k \to 0}$)} & \textbf{866} & \textbf{10.87} & $\mathbf{9.9 \times 10^{-13}}$ \\
\bottomrule
\end{tabular}
\end{small}
\end{table}

\end{document}